%% file: Accessibility.tex
\documentclass[draft, a4paper]{article}
\input{macros}

\input{environs}
\usepackage{leftidx}
\title{A sufficient condition for accessibility of pro-$p$ groups}
\author{Gareth Wilkes}
\begin{document}
\maketitle
\begin{abstract}
We establish a sufficient condition for a finitely generated pro-$p$ group to be accessible in terms of finite generation of the module of ends.
\end{abstract}
\section{Introduction}
One of the key themes in mathematics is the study of an object by breaking it down into `simpler' pieces which may be more amenable to study. In geometric group theory, the common exemplar of this theme is the notion of a {\em graph of groups decomposition} of a group. A necessary part of the power of such a theory is some statement that the pieces into which one decomposes a group really are `simpler'---often realised as a statement that the process of decomposing a group eventually terminates in something that cannot be split further. For graph of groups decompositions, this is expressed via the notion of {\em accessibility}. 

\begin{defn}
Let $G$ be a discrete group. We say $G$ is {\em accessible} if there is a number $n = n(G)$ such that any finite, reduced graph of discrete groups ${\cal G}=(X,G_\bullet)$ with finite edge groups having fundamental group isomorphic to $G$ has at most $|EX| \leq n$ edges.
\end{defn}
The precise definition and notation of a graph of groups decomposition  for discrete groups will not be given here. The definitions for pro-$p$ groups will be given later in Section \ref{sec:AccessBg}.

A torsion free finitely generated group is accessible by Grushko's Theorem. Other important classes of groups known to be accessible are finitely generated groups with a bound on the size of their finite subgroups \cite{Linnell83} and all finitely presented groups \cite{Dunwoody85}. It is known however that not all finitely generated groups are accessible \cite{Dunwoody93}.

Pro-$p$ groups have a theory of graph of groups decomposition which has many features in common with that for discrete groups. In this context it is perhaps natural to define and study accessibility for pro-$p$ groups. The author began this study in \cite{Wilkes19}, and established that---as just seen for discrete groups---not all finitely generated pro-$p$ groups are accessible, but that those with a bound on the size of their finite subgroups are.

We currently lack criteria to establish accessiblity of pro-$p$ groups. In particular it is open whether finitely presented pro-$p$ groups are accessible. In this paper we show a sufficient condition for accessibility in terms of the `module of ends' of an infinite pro-$p$ group as introduced by Korenev \cite{Korenev04}. This criterion parallels that for discrete groups established by Dunwoody \cite{Dunwoody79}. 
\begin{theorem}
Let $G$ be a finitely generated pro-$p$ group. If ${\bf H}^1(G,\Fpof{G})$ is finitely generated as a right $\Fpof{G}$-module then $G$ is accessible. 
\end{theorem}
For the proof, and a more precise statement including an explicit bound for $n(G)$, see Theorem \ref{thm:main}.
\begin{rmk}
While this paper was in preparation, the author learned that this theorem has been independently established by Chatzidakis and Zalesskii \cite{CZ20}. The author is grateful for advance sight of their paper.
\end{rmk}

{\noindent\bf Acknowledgements} The author was supported by a Junior Research Fellowship from Clare College, Cambridge.
\section{Background}
\subsection{The module of ends}

Let $G$ be a pro-$p$ group and let $M$ be a profinite left $G$-module: a profinite abelian group with a continuous $G$-action. We may define the continuous cohomology of $G$ with coefficients in $M$ via the usual formulae: equipping the spaces $C^n(G,M)$ of continuous functions $G^n\to M$ with coboundary operators $C^n(G,M)\to C^{n+1}(G,M)$ and studying images and kernels.

This is a valid definition, but has a problem in the greatest generality: it is not continuous with respect to inverse limits of modules. That is, the groups
\[H^n(G,\varprojlim M_i)\text{ and } \varprojlim H^n(G,M_i)\]
may not be isomorphic in general for an inverse system of profinite modules $(M_i)$.

The problem is that the natural domain of definition of cohomology for profinite groups is with coefficients in discrete torsion modules rather than profinite modules. The addition of a certain finiteness condition FP${}_n$ allows one to define a fully functional theory of cohomology ${\bf H}^k(G,M)$ with coefficients in profinite modules for $k\leq n$, which has the required continuity property
\[{\bf H}^n(G,\varprojlim M_i)\iso \varprojlim {\bf H}^n(G,M_i)\] and agrees with the previous definition where defined. For the purposes of this paper it is only necessary for us to know that finitely generated pro-$p$ groups have property FP${}_1$.

The reader is referred to \cite{SW00} for a full discussion of these matters. We follow \cite{SW00} in using the notation ${\bf H}^k(G,M)$ to emphasize the continuity of the relevant functors.

\begin{defn}
Let $G$ be a finitely generated pro-$p$ group. The {\em module of ends} of $G$ is ${\bf H}^1(G,\Fpof{G})$, where $\Fpof{G}$ is the completed group ring of $G$ over $\F_p$. 
\end{defn}
We refer the reader to \cite[Chapter 5]{RZ00} for definitions and material concerning the completed group ring $\Fpof{G}$ and other profinite modules. 

This definition is ultimately motivated by the classical connection between the topological notion of ends of a discrete group $G$ and the cohomology group $H^1(G,\F_2 G)$ (see, for example, \cite{Stallings68}). For pro-$p$ groups the definition is first found in \cite{Korenev04}, who credits the suggestion to Mel'nikov. The {\em number of ends} of the (finitely generated) pro-$p$ group $G$ is defined to be
\[e(G) = 1-\dim_{\F_p} {\bf H}^0(G,\Fpof{G}) + \dim_{\F_p} {\bf H}^1(G,\Fpof{G}).\]
The second term is either 0 or 1 depending whether $G$ is infinite or finite. The number of ends $e(G)$ is either 0, 1, 2 or infinity \cite[Theorem 1]{Korenev04}. The statement `$G$ has more than one end' is equivalent to `$G$ is infinite and ${\bf H}^1(G,\Fpof{G})$ is non-zero'.

Since $\Fpof{G}$ is a $(G,G)$-bimodule, the group ${\bf H}^1(G,\Fpof{G})$ acquires the structure of a right $\Fpof{G}$-module. Since $G$ is a pro-$p$ group, $\F_p$ is the only simple $G$-module \cite[Lemma 7.1.5]{RZ00}. It follows that $\Fpof{G}$ is a local ring, and that every non-trivial $G$-module surjects $\F_p$. A consequence that will be important for us is recorded below.
\begin{lem}\label{lem:directsumslocalring}
Let $G$ be a pro-$p$ group and let $M$ be a $G$-module generated by $n$ elements. Then $M$ cannot surject a direct sum of more than $n$ non-trivial $G$-modules.
\end{lem}
We will denote by $\dim_{G\text{-}{\sf Mod}}(M)$, respectively $\dim_{{\sf Mod}\text{-}G}(M)$, the minimal size of a generating set of a left, respectively right, $G$-module.
\begin{cnv}
Cohomology theory routinely uses the notation $M^G$ to denote the invariants of a left $G$-module $M$, viz.
\[H^0(G,M) = M^G.\]
In the present paper it is central to the discussion that $M$ is a $(G,G)$-bimodule, and it will be inappropriate to use a notation $M^G$ for the left-invariants of $M$. Instead we will use a left superscript ${}^G\!M$.
\end{cnv}
\subsection{Accessibility}\label{sec:AccessBg}
Here we record the necessary definitions concerning graphs of pro-$p$ groups.
In this paper an (oriented) graph $X$ will consist of a set $VX$ of {\em vertices}, a set $EX$ of {\em edges}, and two functions $d_k\colon EX\to VX$ for $k\in\{0,1\}$ describing the {\em endpoints} of $e$. 
\begin{defn}
A graph ${\cal G}=(X,G_\bullet)$ of pro-$p$ groups consists of a finite graph $X$, a pro-$p$ group $G_x$ for each $x\in X$, and monomorphisms $\bdy_{e,k}\colon G_e\to G_{d_k(e)}$ for each edge $e\in EX$ and for $k=0,1$. A graph of groups is {\em reduced} if $\bdy_{e,k}$ is never an isomorphism when $e$ is not a loop.
\end{defn}
Let ${\cal G}=(X,G_\bullet)$ be a finite graph of pro-$p$ groups. Choose a maximal subtree $T$ of $X$. We define the {\em fundamental pro-$p$ group} $\Pi_1(\cal G)$ of $\cal G$ to be the pro-$p$ group given by the pro-$p$ presentation
\begin{multline*}
\big<G_x\, (x\in X),\, t_e\, (e\in EX) \, \big| \, t_e=1 \text{ (for all }e\in ET)\\ \bdy_{e,0}(g) = g = t_e\bdy_{e,1}(g)t_e^{-1} \text{ (for all } e\in EX, g\in G_e)\big> 
\end{multline*}
The fundamental pro-$p$ group of a graph of groups always exists \cite[Proposition 6.2.1(b)]{Ribes17} and is independent of the choice of $T$ \cite[Theorem 6.2.4]{Ribes17}.

When the graph $G$ has one edge $e$ and one (or two) vertices $u$ (and $v$ respectively), we use the notation $\Pi_1({\cal G})= G_u\amalg_{G_e} (G_v)$ in analogy to classical HNN extensions (amalgamated free products).

For discrete groups, it is automatic that the fundamental group of a graph of groups contains injective copies of the groups $G_x$. For pro-$p$ groups this is no longer necessary and must be enforced as an extra property.
\begin{defn}
A graph of pro-$p$ groups ${\cal G}=(X,G_\bullet)$ is {\em proper} if the natural maps $G_x\to \Pi_1({\cal G})$ are injections for all $x\in X$.
\end{defn}
A proper, reduced graph of pro-$p$ groups whose fundamental pro-$p$ group is isomorphic to a pro-$p$ group $G$ will be referred to as a {\em splitting} of $G$, or a {\em graph of groups decomposition} of $G$.

\begin{defn}
Let $G$ be a pro-$p$ group. We say $G$ is {\em accessible} if there is a number $n=n(G)$ such that any finite, proper, reduced graph of pro-$p$ groups $\cal G$ with finite edge groups having fundamental group isomorphic to $G$ has at most $|EX|\leq n$ edges.
\end{defn}

We require two additional facts concerning graphs of pro-$p$ groups: a structure result and a Mayer-Vietoris sequence.
\begin{prop}\label{prop:virtfree}
Let ${\cal G}=(X,G_\bullet)$ be a proper finite graph of finite $p$-groups. Then $\Pi_1(\cal G)$ is virtually free pro-$p$.
\end{prop}
\begin{proof}
From the definition it is immediate that $G=\Pi_1(\cal G)$ is the pro-$p$ completion of the abstract fundamental group $\Gamma=\pi_1(\cal G)$ of the graph of discrete groups $(X,G_\bullet)$. By properness, all vertex groups inject into $G$, so there is a map $G\to P$ to a finite $p$-group such that all maps $G_x\to G\to P$ are injective. The kernel of this map is the pro-$p$ completion of the kernel of the induced map $\Gamma \to P$, which is a free group by standard Bass-Serre theory \cite{SerreTrees}.
\end{proof}
\begin{prop}[Mayer-Vietoris sequence\footnote{This would be expected to be fairly standard. It requires a proof here for two reasons: the use of the extended functors ${\bf H}^\bullet$ and because the standard reference \cite{Ribes17} deals only with {\em homological} Mayer-Vietoris sequences, these being better behaved in more complex situations when the graph of groups may not be finite.}]
Let $G$ be the fundamental pro-$p$ group of a proper finite graph of pro-$p$ groups ${\cal G} = (X,G_\bullet)$. For a finite $G$-module $M$ there is a natural long exact sequence of the form
\[\cdots H^n(G,M)\to \bigoplus_{v\in VX}H^n(G_v,M)\to \bigoplus_{e\in EX}H^n(G_e,M)\to H^{n+1}(G,M)\cdots\]
If $G$ and all $G_x$ have type FP${}_N$ then for all compact $G$-modules $M$ there is a natural exact sequence
\[\cdots {\bf H}^n(G,M)\to \bigoplus_{v\in VX}{\bf H}^n(G_v,M)\to \bigoplus_{e\in EX}{\bf H}^n(G_e,M)\to {\bf H}^{n+1}(G,M)\cdots\]
terminating with the term $\bigoplus_{v\in VX} {\bf H}^N(G_v,M)$.\end{prop}
\begin{proof}
By \cite[Theorem 6.3.5]{Ribes17}, $G$ acts on a pro-$p$ tree $T$ with $G\lqt T=X$ and with each vertex and edge stabilizer conjugate to some $G_x$. The chain complex of the tree \cite[Section 2.4]{Ribes17} then takes the form
\[0\to \bigoplus_{e\in EX} \Zpof{G/G_e}\to \bigoplus_{v\in VX}\Zpof{G/G_v} \to \Z[p]\to 0.\]
For a finite $G$-module $M$ apply the functor $\Ext^\bullet_\Zpof{G}(-,M)$, which commutes with finite direct sums, and apply the Shapiro Lemma (in the form found in, for example, \cite[Proposition 1.23]{Wilkes19RelCoh}) to find the stated Mayer-Vietoris sequence
\[\cdots H^n(G,M)\to \bigoplus_{v\in VX}H^n(G_v,M)\to \bigoplus_{e\in EX}H^n(G_e,M)\to H^{n+1}(G,M)\cdots\]
Now assume that the finiteness conditions in the statement hold and let $M$ be a profinite $G$-module written as an inverse limit of finite $G$-modules $M=\varprojlim M_i$ \cite[Lemma 5.5.3(c)]{RZ00}. The functors ${\bf H}^n(G_x, -)$ and ${\bf H}^n(G,-)$ are continuous for $n\leq N$ and take finite values on the finite $G$-modules $M_i$. We now take an inverse limit of the Mayer-Vietoris sequences for the $M_i$, noting that the Mittag-Leffler condition guarantees exactness of the inverse limit functor in this case, to obtain the Mayer-Vietoris sequence for ${\bf H}^\bullet(G,M)$.
\end{proof}
\begin{rmk}
Suppose all edge groups $G_e$ are finite, as in the situation of this article. If $G$ has type FP${}_n$ then the first exact sequence in the proposition implies that all vertex groups $G_v$ also have type FP${}_n$ by \cite[Proposition 4.2.3]{SW00}, so we may also use the second exact sequence.
\end{rmk}
\section{Proofs}
\begin{prop}\label{prop:infdirsums}
Let $G$ be a profinite group and let $K$ be a closed subgroup of $G$. If $K$ is a profinite group of type FP$_n$ then 
\[{\bf H}^k(K, \Fpof{G}) \iso {\bf H}^k(K, \Fpof{K})[\![K\lqt G]\!]\]
(as an isomorphism of right $G$-modules) for $k\leq n$.
\end{prop}
\begin{proof}
Take a continuous section $\sigma\colon K\lqt G\to G$ (which exists by \cite[Proposition 2.2.2]{RZ00}), so that we have \cite[Proposition 5.7.1]{RZ00} an isomorphism of $(K,G)$-bimodules 
\[\Fpof{G} \iso \Fpof{K}[\![K\lqt G]\!]\iso \varprojlim \Fpof{K}[K\lqt G/U]\]
where the inverse limit is taken over open normal subgroups $U$ of $G$. Using continuity of the functor ${\bf H}^k(K,-)$, we have
\begin{multline*}{\bf H}^k(K, \Fpof{G}) \iso {\bf H}^k(K,\varprojlim \Fpof{K}[K\lqt G/U])\iso \varprojlim {\bf H}^k(K,\Fpof{K}[K\lqt G/U])\\
\iso \varprojlim {\bf H}^k(K,\Fpof{K})[K\lqt G/U] \iso {\bf H}^k(K, \Fpof{K})[\![K\lqt G]\!] 
 \end{multline*}
as required. To move from the first line to the second,  note that cohomology commutes with finite direct sums in the second entry. 
\end{proof}

\begin{lem}
Let $G$ be a profinite group and let $K\leq G$. Then 
\[{\bf H}^0(K, \Fpof{G})\iso\begin{cases}
0 & \text{ if $K$ is infinite}\\ N_K\Fpof{G} \iso \Fpof{K\lqt G}& \text{ if $K$ is finite}
\end{cases} \]
where $N_K = \sum_{k\in K} k$.
\end{lem}
\begin{proof}
If $K$ is infinite then this result is deduced from Proposition \ref{prop:infdirsums} via \cite[Lemma 3]{MS02}\footnote{This is the citation as given in \cite{Korenev04}. The author was unable to locate the paper \cite{MS02}, so for convenience we will note that a proof may also be found in \cite[Proposition B.1]{Wilkes19RelCoh}}. If $K$ is finite then take an open normal subgroup $U$ of $G$ such that $U\cap K = 1$ and note that 
\[H^0(K, \F_p[G/U]) = {}^K\!\F_p[G/U] = N_K\F_p[G/U].\]
Then write 
\[{\bf H}^0(K, \Fpof{G}) \iso \varprojlim {\bf H}^0(K,\F_p[G/U]) = \varprojlim N_K\F_p[G/U] = N_K \Fpof{G} \]
where the limit is taken over open normal subgroups $U$ of $G$ such that $U\cap K=1$.
\end{proof}
\begin{lem}
If $G$ is finite then $H^1(G,\F_p[G])=0$.
\end{lem}
\begin{proof}
This follows from Shapiro's Lemma, since for $G$ finite the induced and coinduced modules are isomorphic.
\end{proof}
We will require the following two propositions.
\begin{prop}[\cite{Korenev04}, Corollary 3]
Let $G$ be a finitely generated free pro-$p$ group. Then \[{\bf H}^1(G,\Fpof{G})\neq 0.\]
\end{prop}
\begin{prop}[\cite{Korenev04}, Lemma 2]
Let $G$ be a finitely generated pro-$p$ group and let $U$ be an open subgroup of $G$. Then ${\bf H}^1(G,\Fpof{G})\iso {\bf H}^1(U,\Fpof{U})$.
\end{prop}
In particular, a finitely generated infinite virtually free pro-$p$ group is not one-ended.
\begin{prop}\label{prop_more}
Let $G$ be a finitely generated pro-$p$ group with a non-trivial proper reduced graph of groups decomposition $(X,G_\bullet)$ with finite edge groups. Then ${\bf H}^1(G,\Fpof{G})\neq 0$.
\end{prop}
\begin{proof}
It suffices, by collapsing \cite[Proposition 2.8]{Wilkes19} a portion of the finite graph of groups, to consider one-edge splittings.  We apply a Mayer-Vietoris sequence, of which the relevant portion is:
\[0 \to \bigoplus_{v\in VX} {\bf H}^0(G_v,\Fpof{G}) \to {\bf H}^0(G_e,\Fpof{G}) \to {\bf H}^1(G,\Fpof{G}) \]
where we note that $G$ is infinite so that ${\bf H}^0(G,\Fpof{G})=0$. We derive a contradiction from the assumption ${\bf H}^1(G,\Fpof{G})= 0$. 

If $X$ is a loop with an edge $e$ and vertex $v$ then either $G_v$ is finite, so that $G$ is virtually free (Proposition \ref{prop:virtfree}) and thus not one-ended, or $G_v$ is infinite and $G_e$ finite, so that the map 
\[0={\bf H}^0(G_v,\Fpof{G}) \to {\bf H}^0(G_e,\Fpof{G})\neq 0\]
cannot possibly be surjective.

If $X$ has a single edge $e$ and two vertices $v$ and $w$, then either $G_v$ and $G_w$ are both finite---whence again $G$ is virtually free---or without loss of generality $G_w$ is infinite, hence we have a surjection
\[{\bf H}^0(G_v,\Fpof{G}) \to {\bf H}^0(G_e,\Fpof{G}) = N_{G_e}\Fpof{G}.\]
It follows that $N_{G_e} \in N_{G_v}\Fpof{G}$ is (left-)$G_v$-invariant, so that $G_v=G_e$ contradicting the reduced hypothesis.
\end{proof}
We will need the following graph theoretic lemma.
\begin{lem}\label{lem_counting} Let $X$ be a graph. Let $M(X)$ denote the size of a maximal matching in $X$: a maximal set $\{e_i\}$ of edges of $X$ such that $e_i$ and $e_j$ do not share an endpoint if $i\neq j$. Let $T(X)$ denote the quantity
\[T(X) = \#(\text{leaves in }X)  - \chi(X).\]
Then \[|EX| \leq 2M(X) + 9 T(X).\]
\end{lem}
\begin{proof}
Let the full subgraph of $X$ spanned by valence 2 vertices have components $S_1,\ldots, S_r$, each of which is a line segment. Forming a matching solely from edges of the $S_i$ we find a bound
\[M(X) \geq \sum_i \left\lfloor \frac{|ES_i|}{2}\right\rfloor \geq \sum_i \frac{|ES_i|-1}{2}\]
and hence
\[2M(X)+r\geq \sum_i |ES_i| \geq |EX| - \sum_{\val(v)\neq 2} \val(v).\]

Now consider the graph $Y$ obtained from $X$ by ignoring the valence 2 vertices. Note that $T(Y)=T(X)$, $r\leq |EY|$ and that 
\[\sum_{v\in VX : \val(v)\neq 2} \val(v) = 2 |EY|.\]
Next, $|EY|$ is bounded in terms of $T(Y)$:
\begin{eqnarray*}
|EY| & = & |VY| - \chi(Y)\\
&=& T(Y) + |VY| - \#(\text{leaves}) \\
&=& T(Y) + |\{v : \val(v)\geq 3\}|\\
&\leq & T(Y) +  \frac{1}{3}\sum_{\val(v)\geq 3} \val(v)\\
&\leq & T(Y) + \frac{2}{3} |EY|.
\end{eqnarray*}
So finally we have
\[|EX| \leq 2M(X)+r+\sum_{\val(v)\neq 2} \val(v)\leq 2M(X) + 3|EY|\leq 2M(X)+9T(X). \]
\end{proof}

\begin{theorem}\label{thm:main}
Let $G$ be a finitely generated pro-$p$ group such that the module of ends ${\bf H}^1(G,\Fpof{G})$ is finitely generated as a right $G$-module. Then $G$ is accessible. In particular, if $(X,G_\bullet)$ is a finite reduced graph of groups decomposition of $G$ with finite edge groups then
\[|EX|\leq 2\dim_{{\sf Mod}\text{-}G}({\bf H}^1(G,\Fpof{G})) + 9(b_1(G)-1).\]
\end{theorem}
\begin{proof}
Let $(X,G_\bullet)$ be a finite reduced graph of groups decomposition of $G$ with finite edge groups. Let $M$ be a maximal matching in $X$. By collapsing each $e_i\in M$ we obtain \cite[Proposition 2.8]{Wilkes19} a new graph of groups decomposition $(X',G'_\bullet)$ for $G$ with finite edge groups, such that each vertex $v\in X'$ deriving from an edge $e_i\in M$ has associated vertex group $G'_v = G_{d_0(e_i)} \amalg_{G_{e_i}} (G_{d_1(e_i)})$ with more than one end by Proposition \ref{prop_more}. The Mayer-Vietoris sequence corresponding to the new splitting $(X',G'_\bullet)$ then includes a surjection
\[{\bf H}^1(G,\Fpof{G}) \to \bigoplus_{v\in VX'}{\bf H}^1(G'_v,\Fpof{G})\to 0\]
where the right hand side has at least $|M|$ non-zero summands. It follows by Lemma \ref{lem:directsumslocalring} that 
\[M(X)=|M|\leq \dim_{{\sf Mod}\text{-}G}({\bf H}^1(G,\Fpof{G})).\]

We now also bound the quantity $T(X)$. If the vertex $v\in VX$ is a leaf, then there is a homomorphism $\phi_v\colon G_v\to \F_p$ which kills the single edge group adjacent to $v$. By killing all the vertex groups which are not leaves, the kernels of the $\phi_v$, and all edge groups, we find a surjection from $G$ to the fundamental group of a graph of groups which has vertex group $\F_p$ on each leaf and trivial groups otherwise. This fundamental group is readily seen to be
\[F_{1-\chi(X)} \amalg \coprod_{v \text{ leaf}} \F_p\]
where $F_r$ denotes a free pro-$p$ group of rank $r$. We deduce that
\[b_1(G) \geq \#(\text{leaves of }X) + 1-\chi(X) = T(X)+1.\]
The result now follows from Lemma \ref{lem_counting}.
\end{proof}

We observe that finite generation of the module of ends is not an obviously strictly stronger condition than accessibility: in the world of finitely generated discrete groups it is equivalent to accessibility by a theorem of Dunwoody \cite[Theorem 5.5]{Dunwoody79}. This theorem relies in an essential way on Stallings' Theorem on splittings of groups. We include the argument below for completeness, and in order to note that it would be valid for pro-$p$ groups if a pro-$p$ Stallings' Theorem were true---but such a theorem is not known at present.

\begin{prop}
Let $G$ be a finitely generated discrete group. If $G$ is accessible then $H^1(G,\F_2G)$ is finitely generated as a right $G$-module.
\end{prop}
\begin{proof}
Let $(X,G_\bullet)$ be a reduced graph of groups decomposition of $G$ with finite edge groups and with the maximum possible number of edges. Note that each vertex group is finitely generated. Each vertex group $G_v$ must be one-ended: else by Stalling's Theorem \cite{Stallings70} we could split $G_v$ over a finite subgroup, and thus find a graph of groups decomposition of $G$ with one additional edge (noting that all edge groups adjacent to $G_v$ are finite, and hence lie in a vertex group of the splitting of $G_v$). Therefore $H^1(G_v,\F_2G)=0$ for all $G_v$, and the Mayer-Vietoris sequence gives a surjection 
\[\bigoplus_{e\in EX} \F_2[G_e\lqt G] \twoheadrightarrow H^1(G,\F_2G)\]
whence $H^1(G,\F_2G)$ is finitely generated as a right $G$-module. 
\end{proof}

\bibliographystyle{alpha}
\bibliography{Accessibility}

\end{document}

%% file: macros.tex
\usepackage{amsmath}
\usepackage{amsfonts}
\usepackage{amssymb}
\usepackage{amsthm}
\usepackage{mathtools}
\usepackage{tikz}
\usepackage{tikz-cd}
\usepackage{enumerate}
\usepackage{layout}
\usepackage{mathabx}
\usepackage{bm}
\usetikzlibrary{patterns}

\DeclareMathOperator{\Ext}{Ext}

\DeclareMathOperator{\val}{val}

\newcommand{\bdy}{\ensuremath{\partial}}

\newcommand{\iso}{\ensuremath{\cong}}
\newcommand{\Z}[1][]{\ensuremath{\mathbb{Z}_{#1}}}

\newcommand{\F}{\ensuremath{\mathbb{F}}}

\newcommand{\Zpof}[1]{{\ensuremath{\Z[p][\![#1]\!]}}}
\newcommand{\Fpof}[1]{{\ensuremath{\F_p[\![#1]\!]}}}

\newcommand{\lqt}{\backslash}

%% file: environs.tex
\newtheorem{theorem}{Theorem}[section]

\newtheorem{prop}[theorem]{Proposition}

\newtheorem{lem}[theorem]{Lemma}

\theoremstyle{definition}
\newtheorem{defn}[theorem]{Definition}
\newtheorem*{cnv}{Conventions}

\theoremstyle{remark}
\newtheorem*{rmk}{Remark}

\theoremstyle{plain}

\newcounter{introthmcount}
\theoremstyle{definition}